\documentclass[oneside,english,reqno]{amsart}

\usepackage{algorithm}
\usepackage{algorithmic}
\usepackage{multirow}

\usepackage{enumitem}

\usepackage[T1]{fontenc}
\usepackage[latin9]{inputenc}
\usepackage{babel}
\usepackage{amsmath}
\usepackage{amsthm}
\usepackage{amssymb}

\usepackage{comment}

\usepackage{graphicx}

\usepackage{graphicx}
\usepackage{tikz}
\usetikzlibrary{shapes.misc}
\usepackage{float}

\usepackage[unicode=true,pdfusetitle,
bookmarks=true,bookmarksnumbered=false,bookmarksopen=false,
breaklinks=false,pdfborder={0 0 1},backref=false,hidelinks]
{hyperref}

\makeatletter

\numberwithin{equation}{section}
\numberwithin{figure}{section}

\makeatother

\newtheorem{theorem}{Theorem}

\newtheorem{definition}{Definition}
\newtheorem{example}[theorem]{Example}

\newtheorem{remark}{Remark}

\newtheorem{assumption}{Assumption}

\begin{document}

\title[Finding global solutions for a class...]{Finding global solutions for a class of possibly nonconvex QCQP problems through the S-lemma
}


\author{
	Ewa M. Bednarczuk$^{1}$
}
\author{
	Giovanni Bruccola$^2$
}

\thanks{$^1$ Systems Research Institute, Polish Academy of Sciences, Newelska 6, 01-447 Warsaw
              \email{Ewa.Bednarczuk@ibspan.waw.pl}      
}
\thanks{$^2$  Systems Research Institute, Polish Academy of Sciences, Newelska 6, 01-447 Warsaw\\
              the research is supported by the ITN Marie-Curie Project TraDe-Opt
               \email{Giovanni.Bruccola@ibspan.waw.pl}
}

\begin{abstract}
In this paper we provide necessary and sufficient (KKT) conditions for global optimality for a new class of possibly nonconvex quadratically constrained quadratic programming (QCQP) problems, denoted by S-QCQP. The class consists of QCQP problems where the matrices of the quadratic components are formed by a scalar times the identity matrix.
Our result relies on a generalized version of the S-Lemma, stated in the context of general QCQP problems. 
Moreover, we prove the exactness of the SDP and the SOCP relaxations for S-QCQP. 
\end{abstract}
\maketitle

\section{Introduction}
\label{intro}
The aim of this work is to provide a  characterization for the global minima of a specific class of Quadratically Constrained Quadratic Programming (QCQP) problem.

A generic (QCQP) problem is defined as in \cite{Boyd1}:
\begin{equation}
    \label{QCQP1}
    \tag{QCQP}
    \begin{split}
        &{Minimize\,}_{x\in \mathbb{R}^n} \ \ J(x):=x^TA_Jx+2b_J^Tx+c_J\\
        &s.t.\ \ f_k(x):=x^TA_kx+2b_k^Tx+c_k \le 0,\ \ k=1,...,m
    \end{split}
\end{equation}
with $x\in\mathbb{R}^{n}$, $A_J,A_k \in S^{n}$,where $S^n$ denotes the space of symmetric $n \times n$ matrices, $b_J,b_k \in \mathbb{R}^n $ and $c_J,c_k \in \mathbb{R}$ are given data for $k=1,...,m$.

We focus on \eqref{QCQP1} problems which satisfy the following assumption:
\begin{assumption}
\label{assumaI}
 The matrices $A_J,A_{k}$ $\forall\,k\in\{1,...,m\}$ are of the form $A_J=a_JI,A_k=a_kI$ where $I$ is the $n\times n$ identity matrix, and $a_J,a_k\in\mathbb{R}$. 
\end{assumption}

When \autoref{assumaI} holds,  \eqref{QCQP1} takes the form
\begin{equation}
    \label{QCQP2}
    \tag{S-QCQP}
    \begin{split}
        &{Minimize\,}_{x\in \mathbb{R}^n} \ \ J(x):=a_Jx^Tx+2b_J^Tx+c_J\\
        &s.t.\ \ f_k(x):=a_kx^Tx+2b_k^Tx+c_k \le 0,\ \ k=1,...,m
    \end{split}
\end{equation}

Since in \eqref{QCQP2} the matrices $A_J,A_k$, $k=1,...,m$ do not appear and we only have the scalars $a_J,a_k$, $k=1,...,m$, we propose to denote \eqref{QCQP2} as \textit{scalar QCQP}, i.e. S-QCQP.

\begin{remark}
    The scalars $a_J,a_k$, $k=1,...,m$ can be also equal to zero for \eqref{QCQP2}.
    When $a_J=0$ or $a_k=0$ for some $k$, the corresponding $J(x)$ or $f_k(x)$ are linear forms. 
\qed
\end{remark}

Consider the following KKT conditions:
\begin{equation}
    \label{kkt0}
    \tag{KKT}
    \begin{split}
        &(i)\ \ \ \nabla(J+\sum\limits_{k=1}^m\gamma_k f_k)(x^*)=0\\
        &(ii)\ \ \gamma_kf_k(x^*)=0\ \ k\in\{1,...,m\}\\
        &(iii)\ A_J+\sum\limits_{k=1}^m\gamma_kA_k\succeq0
    \end{split}
\end{equation}

The main result of this work is to prove that \eqref{kkt0} are necessary and sufficient optimality conditions for \eqref{QCQP2}, if $m+1< n$, there exists $x_0$ such $f_k(x_0)<0\,\forall\,k\in\{1,...,m\}$ and if the following assumption holds.

\begin{assumption}
\label{preliminaryassumpt}
Assume that there exists a global minimum $x^*$ of problem \eqref{QCQP1} and there exists $\gamma\in\mathbb{R}^m_+\backslash\textbf{0}_m$ such that $A_J+\sum\limits\gamma_kA_k\succeq0$.
\end{assumption}

The assumption $\exists\,\gamma\in\mathbb{R}^m_+\backslash\textbf{0}_m$ such that $A_J+\sum\limits\gamma_kA_k\succeq0$, is a standard assumption in the literature related to QCQP, but it is non verified by nonconvex QP (quadratic problems with linear constraints only), since $A_J\nsucceq0$. 
KKT conditions and the SDP relaxation for QP are studied for example in \cite{gon1}, \cite{burer2}.

The assumption that there exists a global minimum $x^*$ of problem \eqref{QCQP2}, allows us to consider a heterogeneous set of cases, like the ones of following examples.
\begin{itemize}
    \item \eqref{QCQP2} with nonconvex quadratic and linear constraints, when the Hessian of the objective function $J$ is positive definite.
    \item \eqref{QCQP2} where the objective function has a negative definite Hessian and convex quadratic and linear constraints.
\end{itemize}

The KKT conditions for \eqref{QCQP1} with one convex quadratic constraint are studied for example in \cite{lucidi1}, \cite{sorensen1}.
In paper \cite{jeya1}, the authors prove that \eqref{kkt0} are necessary and sufficient for \eqref{QCQP1} such that $\forall\,k\in\{J,1,...,m\}$ the matrices $H_k:=\begin{pmatrix}
A_k&b_k\\b_k^T&c_k
\end{pmatrix}$ are $Z$-matrices (which are matrices with non positive off diagonal elements), if there exists $x_0$ such that $f_k(x_0)<0\,\forall\,k\in\{1,...,m\}$.

In order to characterize the optimal value of Z-matrices QCQP, in \cite{jeya1} the authors propose a generalized version of the S-Lemma, in the sense described in Section 2-Preliminaries.

We apply the approach proposed in \cite{jeya1} for a specific subclass of \eqref{QCQP1} to prove that the conditions \eqref{kkt0} are necessary and sufficient optimality conditions for \eqref{QCQP2}, . 

The main ingredients used in deriving the result of the present paper are a generalized version of the S-Lemma and the convexity of the set 
$$
\Omega_0:=\{(f_0(x),f_1(x),...,f_m(x))|x\in\mathbb{R}^n\}+int\mathbb{R}^{m+1}_+\text{ with }f_0:=J(x)-J(x^*).
$$

Then we prove the legitimacy of this approach in the context of \eqref{QCQP2}.

The paper \cite{polik1} provides a comprehensive survey on the famous S-Lemma, while \cite{jeya2} explores the relations between the S-Lemma and the Lagrangian multipliers of \eqref{QCQP1}.

The S-Lemma was the first important results related to the more general S-procedure described in \cite{fradkov1}.
These important theorems are treated from an historical point view in \cite{gusev1}.
Under data uncertainty, \cite{barro1} applies a S-Lemma based approach similar to the one proposed by \cite{jeya1} and us.
The paper \cite{RZ1} extend the results of \cite{jeya1} to Z-matrices \eqref{QCQP1} infinite number of inequalities.

We will exploit \eqref{kkt0} to prove that the convex relaxations (SDP) and (SOCP) of \eqref{QCQP2} are exact. 
We say that a relaxation is exact when the objective values at a global minimum of the original problem and of the relaxation coincide.

Hence, in order to find the optimal value of \eqref{QCQP2} problems, we can use solvers capable to deal with (SOCP) and (SDP) programs, such as the ones described in \cite{Doma1}, \cite{Sturm1},\cite{Kim1}. 

Even if the (SDP) is not exact for a \eqref{QCQP1}, it still provides a lower bound for the optimal value of the \eqref{QCQP1}, \cite{Anstrei1}. 

Many papers and books are dedicated to the convex relaxations of \eqref{QCQP1} problem and, in particular, the semi-definite programming (SDP) relaxation and its modifications
(see e.g. \cite{Vanden1}, \cite{blek1}, \cite{loc0}, \cite{burer1}).

The algorithms proposed in literature often solves exactly the RLT-SDP relaxation of a QCQP, which improves the approximation of the solution given by the SDP relaxation, \cite{Anstrei1}, \cite{bao1}.
The papers of \cite{Luo2} and \cite{ello1} propose algorithms which provides a tighter lower bound to \eqref{QCQP1} then the SDP relation. 
The algorithm proposed by \cite{ello1}, can deal also with a MIQCP (mixed-integer quadratically constrained program). The MIQCP can be defined by adding the following constraints to \eqref{QCQP1}: taking $I:={1,...,n}$ and $J \subset I$, $x_i \in \mathbb{N}\ \ \forall\, i \in J$ and $x_i \in \mathbb{R}^n\ \ \forall\, i \in I\backslash J$. 
Paper \cite{furini1} offers an exhaustive survey on solvers which treat different types of QCQP and MIQCP, notwithstanding the elevated obsolescence rate of this kind of works.

In papers \cite{ben1} and \cite{locatelli1}, the authors study the SDP relaxation in the framework of SD QCQP (see section 3.3), i.e. when all the matrices $A_i$, $i=\{0,...,m\}$ which appear in \eqref{QCQP1} are SD (simultaneously diagonalizable), see section 6. 
In this framework, the SDP relaxation is equivalent to other convex relaxations as the Second Order Cone Programming (SOCP) relaxation, see \cite{Ali1}, \cite{Kim1}, \cite{Lobo1}, \cite{locatelli2}.

The organization of the paper is as follows.
In Section 2, we provide the preliminaries concerning the notation, the S-Lemma, the Fermat rule and the convex separation theorem which are used in the sequel.

The main result of Section 3 is \autoref{globalminchar}, which provides the global minima characterization for general \eqref{QCQP1} problems in the form of \eqref{kkt} conditions.

In Section 4 we apply \autoref{globalminchar} to provide a characterization of the global solution for \eqref{QCQP2}.
The main result is \autoref{omega0isconvex} which proves the convexity of the set $\Omega_0$.

In Section 5 we discuss the exactness of the SDP and the SOCP relaxation for \eqref{QCQP2}, with the help of the equivalent convex relaxation provided by \cite{ben1}.

In Section 6, we make a comparison with the result provided in Section 4 and the results on global minima characterization which can be found in literature for \eqref{QCQP1} when the number of constraints $m=2$. Finally, we solve the (KKT) condition for \eqref{QCQP2} with $m=2$ and $n >> m$.

\section{Preliminaries}

Given a vector $x\in\mathbb{R}^n$, we say $x\ge0$ when, $(\forall\,i\in\{1,...,m\})$, $x_{i}\ge0$, $\mathbb{R}^n_{+}:=\{x\in\mathbb{R}^{n}\ |\ x\ge 0\}$. 
$\mathbf{0}_n$ denotes the all-zero vector in $\mathbb{R}^n$.
Given two vectors $x,y\in\mathbb{R}^n$, $\langle x,y\rangle:=\sum\limits_{i=1}^nx_iy_i$ denotes the inner product between $x$ and $y$ and $\langle x,y\rangle=x^Ty=y^Tx$.
The corresponding norm is the Euclidean norm denoted by $\|\cdot\|=\sqrt{\langle\cdot,\cdot\rangle}$.

Let $S^n$ be the set of symmetric matrices in $\mathbb{R}^{n\times n}$.
$S^n_+$ and $S^n_{++}$ are the cones of symmetric matrices which are also positive semidefinite and positive definite, respectively. 
If a matrix $A$ belongs to $S^n_+$ then we write $A\succeq0$; if $A\in S^n_{++}$ then we write $A\succ0$.

The general quadratic functional is of the form
\begin{equation}
    \label{fuctional}
    f_k(x):=x^TA_kx+2b_k^Tx+c_k
\end{equation}
where $A_k\in S^n$, $b_k\in\mathbb{R}^n$ and $c_k\in\mathbb{R}$; $k=1,...,m$ are the indices which refer to the constraints of \eqref{QCQP1}.

$diag(A)\in\mathbb{R}^n$ is the vector of the elements in the main diagonal of $A\in\ S^n$.
The \textit{trace inner product} $\langle A, B\rangle$,
between  symmetric matrices $A$, $B$  of dimension $n\times n$ is defined as $\langle A, B\rangle:=Tr(B^{T}A)=\sum_{i=1}^{n}\sum_{j=1}^{n}a_{ij}b_{ij}$.
 Let $|\cdot|$ denote the cardinality of a set.
 $C$ is an affine subspace if $C\neq\emptyset$ and $\forall\,\lambda\in\mathbb{R}$ and for all distinct $x,y\in C$
 \begin{equation*}
     \label{affine}
     \lambda x + (1-\lambda) y \in C
 \end{equation*}
 $\textit{aff}\, C$ denotes the affine hull of $C$, i.e. the smallest affine subspace of $\mathbb{R}^n$ containing $C$.
 $B:=\{x\,|\,\|x\|\le1\}$ is the \textit{Euclidean unit ball} in $\mathbb{R}^n$.
 $int\,C$ denotes the interior of $C$, 
 $$
int\,C:=\{x\in C\,|\,(\exists\,\epsilon>0)\,(x+\epsilon B)\subset C\}
 $$
 $ri\,C$ is the relative interior of $C$, i.e. :
 \begin{equation}
     \label{riC}
     ri\,C:=\{x\in \textit{aff}\,C\,|\,(\exists\,\epsilon>0)\,(x+\epsilon B)\cap(\textit{aff}\,C)\subset C\}
 \end{equation}

An important theorem for the optimality conditions of \eqref{QCQP1} is the S-lemma.
We recall a generalized version of the S-Lemma that can be found in \cite{polik1}.

\begin{theorem} (Yakubovich  S-Lemma)
\label{Slemma}
Let $f,g: \mathbb{R}^{n} \to \mathbb{R}$ be quadratic functionals of the form \eqref{fuctional} and suppose that there exists a point $x_0\in \mathbb{R}$ such that $g(x_0)< 0$. The following statements (i) and (ii) are equivalent.
\begin{equation}
\label{fgminus0}
\tag{i}
    \begin{split}
    &(\nexists\, x \in \mathbb{R}^n)\ \ s.t.\\
    &f(x) < 0 \\ &g(x) \le 0
    \end{split}
\end{equation}
\begin{equation}
    \label{lamb}
    \tag{ii}
    \exists \gamma \ge 0 \,\,\,such\,\,\,that\,\,\, f(x)+\gamma g(x) \ge 0 \,\,\, \forall x \in \mathbb{R}^{n}
\end{equation}
\end{theorem}

\begin{definition}
\label{genSlemma}
Consider a collection of quadratic functionals $f_k: \mathbb{R}^{n} \to \mathbb{R}$ $(k=1,...,m)$ of the form \eqref{fuctional}.
A theorem of the alternative is called generalized version of the S-Lemma if it establishes under which assumptions on the functionals $f_k$ only one between the following statements holds:
\begin{enumerate}
    \item $\exists\,x\in\mathbb{R}^n$ such that $ f_k(x)<0\ \ \forall\,k\in\{1,...,m\}$
    \item $(\exists\,\gamma\in\mathbb{R}_+^m\backslash\textbf{0}_m)$ $\sum\limits_{k=1}^m\gamma_kf_k(x)\ge0$ $\forall\,x\in\mathbb{R}^n$
\end{enumerate}
\end{definition}

Another important result is as follows.
\begin{theorem}
\label{fermat}
(Fermat necessary optimality conditions)
Assume that $f:\mathbb{R}^n\rightarrow\mathbb{R}$ is continuously differentiable on an open set $D\subset\mathbb{R}^n$. Then,
\begin{itemize}
    \item if $x^*\in D$ is a local minimizer of $f$, then it must verify $\nabla f(x^*)=0$.
    \item if $f$ is twice continuously differentiable, then we also have $\nabla^2f(x^*)$ positive semidefinite.
\end{itemize}
\end{theorem}
Moreover, in the convex case \autoref{fermat} can be rewritten in the following form.
\begin{theorem}
\label{fermatconv}
(Fermat necessary and sufficient optimality conditions)
Assume that $f:\mathbb{R}^n\rightarrow\mathbb{R}$ is convex and continuously differentiable on an open set $D\subset\mathbb{R}^n$. Then, $x^*\in D$ is a global minimizer of $f$ if and only if
$$
\nabla f(x^*)=0.
$$
\end{theorem}

Let $C_1$ and $C_2$ be non-empty sets in $\mathbb{R}^n$.

\begin{definition}
\label{hypersep}
(\cite{rocka}, section 11)
\begin{itemize}
    \item A hyperplane $H$ is said to \textit{separate} $C_1$ and $C_2$ if $C_1$ is contained in one of the closed half spaces associated to $H$ and $C_2$ is contained in the opposite closed half space.
\item $H$ is said to \textit{separate properly} $C_1$ and $C_2$ if they are not both contained in $H$ itself.
\end{itemize}
\end{definition}

We are ready to state the convex separation theorem that will be crucial in the next section.

\begin{theorem}
\label{convexsept}
(\cite{rocka}, Theorem 11.3)

Let $C_1$ and $C_2$ be non-empty convex sets in $\mathbb{R}^n$. 
In order that there exists a hyperplane that separates $C_1$ and $C_2$ properly, it is necessary and sufficient that $ri\,C_1$ and $ri\,C_2$ have no point in common. 
\end{theorem}

For $n$-dimensional convex sets in $\mathbb{R}^n$, $\textit{aff}\,C=\mathbb{R}^{n}$  and so, by \eqref{riC}, we have $ri\,C=int\,C$ (\cite{rocka}, section 6).

Hence, we can rewrite \autoref{convexsepth} as follows.
\begin{theorem}
\label{convexsepth}

Let $C_1$ and $C_2$ be full dimensional non-empty convex sets in $\mathbb{R}^n$. 
In order that there exists a hyperplane that separates $C_1$ and $C_2$ properly, it is necessary and sufficient that $int\,C_1$ and $int\,C_2$ have no point in common. 
\end{theorem}

\section{Global Minima Characterization for general (QCQP)}
\label{global:characterisation}

In this section we characterize global minima of  \eqref{QCQP1} problem by  \eqref{kkt} conditions derived with the help of a generalized form of the S-Lemma as defined in \autoref{genSlemma}.

Our approach is inspired by the one proposed in \cite{jeya1} to characterize the global minima of $Z$-matrices \eqref{QCQP1}, i.e. \eqref{QCQP1} with the matrices
\begin{equation}
\label{matrix:H}
H_k:=\begin{pmatrix}A_k&b_k\\b_k^T&c_k\end{pmatrix}\ \ k=1,...,m\text{ and }H_J:=\begin{pmatrix}A_J&b_k\\b_J^T&c_J\end{pmatrix}
\end{equation}
having all the off diagonal elements non positive.

In contrast to \cite{jeya1}, we are not in the framework of $Z$-matrices \eqref{QCQP1}, but we consider the general case of \eqref{QCQP1} problems.

In the sequel, we take into account the following additional assumption.

\begin{assumption}
\label{omegaconv}
Consider a collection of quadratic functionals $f_k(x)=x^TA_kx+b_k^Tx+c_k$ $k=0,...,m$. 
The set $\Omega_0$, with 
\begin{equation}
    \label{omega}
    \Omega_0:=\{(f_0(x),f_1(x),...,f_m(x))\ |\ x\in\mathbb{R}^n\}+int\mathbb{R}^{m+1}_+
\end{equation}
 is convex.
\end{assumption}

We prove a generalized version of the S-Lemma, in the form of a theorem of the alternatives.

\renewcommand{\theenumi}{\roman{enumi}}%
\begin{theorem}
\label{opSlemma}

If \autoref{omegaconv} holds, then
exactly one of the following statements is valid:
\begin{enumerate}
    \item $\exists\,x\in\mathbb{R}^n$ such that $ f_k(x)<0\ \ \forall\,k\in\{0,...,m\}$
    \item $(\exists\,\gamma\in\mathbb{R}_+^{m+1}\backslash\textbf{0}_{m+1})$ $\sum\limits_{k=0}^m\gamma_kf_k(x)\ge0$ $\forall\,x\in\mathbb{R}^n$
\end{enumerate}
\end{theorem}
\renewcommand{\theenumi}{\arabic{enumi}}%
\begin{proof}
The implication [Not(ii)$\Rightarrow$(i)] is immediate (by contradiction).
To show the implication [Not(i)$\Rightarrow$(ii)], assume that (i) does not hold, i.e., the system 
\begin{equation}
\label{slater1}
f_k(x)<0\ \ \forall\,k\in\{0,...,m\}
\end{equation}
has no solution.
By the definition of $\Omega_0$ in \eqref{omega}, the inconsistency of the system \eqref{slater1} implies that 
\begin{equation} 
\label{separationzero}
\Omega_0\cap\ (-\text{int} \mathbb{R}_+^{m+1})=\emptyset.
\end{equation}
To see this, suppose by contrary,  that there exists $y\in\Omega_{0}\cap (-\text{int}\mathbb{R}_{+}^{m+1})$. By the definition \eqref{omega} of  $\Omega_{0}$ (for functions $f_{k}$, $k=0,...,m$), there exist $x_{0}\in\mathbb{R}^{n}$,  $C_{0}, C_{1}\in\text{int}\mathbb{R}_{+}^{m+1}$ such that
$$
y=(f_{0}(x_{0}),...,f_{m}(x_{0}))+C_{0}=-C_{1}\in
(-\text{int}\mathbb{R}_{+}^{m+1}),
$$
i.e., $(f_{0}(x_{0}),...,f_{m}(x_{0}))=-C_{0}-C_{1}\in(-\text{int}\mathbb{R}_{+}^{m+1})$ contradictory to \eqref{slater1}. This proves \eqref{separationzero}.

 Since $\Omega_0$ and $-\text{int}\mathbb{R}_{+}^{m+1}$ are full dimensional in $\mathbb{R}^{m+1}$, non-empty and convex, by \eqref{separationzero}, 
$$ 
int\,\Omega_0\cap\ \,(-\text{int} \mathbb{R}_+^{m+1})=\emptyset,
$$
 we can apply \autoref{convexsepth}. 
 So there exists a hyperplane which separates $\Omega_0$ and $-\text{int}\mathbb{R}_{+}^{m+1}$ properly, i.e. $\gamma\in\mathbb{R}^m\backslash\textbf{0}_m$ such that
\begin{equation} 
\label{separation}
\sum_{k=0}^{m} \gamma_{k}y_{k}\ge 0\ \ \forall\ y\in\Omega_{0}
\end{equation}
and $\sum_{k=0}^{m} \gamma_{k}y_{k}\le 0\ \ \forall\ y\in(-int \mathbb{R}_+^{m+1}).$ This latter inequality shows that it must be $\gamma\in\mathbb{R}_{+}^{m+1}\backslash\textbf{0}_{m+1}$.

Consequently, by the definition \eqref{omega} of  $\Omega_{0}$, for $k=0,...,m$
$$
y=(f_{0}(x),...,f_{m}(x))+C\in\Omega_{0},\ \ \ C\in\text{int} \mathbb{R}_{+}^{m+1}.
$$
By this, and the formula \eqref{separation}, we get 
\begin{equation} 
\label{separation1}
\sum_{k=0}^{m} \gamma_{k}(f_{k}(x)+C_{k})\ge 0\ \ \forall\ x\in\mathbb{R}^{n}, \ \ C_k\in\text{int}\mathbb{R}_+.
\end{equation}
Consequently, it must be
\begin{equation} 
\label{separation2}
\sum_{k=0}^{m} \gamma_{k}f_{k}(x)\ge 0\ \ \forall\ x\in\mathbb{R}^{n}.
\end{equation}
Otherwise, 
$
\sum_{k=0}^{m} \gamma_{k}f_{k}(\bar{x})< 0\ \ \text{for some  }\ \bar{x}\in\mathbb{R}^{n},$ and it would be possible to choose $C\in\text{int}\mathbb{R}_{+}^{m+1}$ with components $C_{k}>0$ small enough so as 
$$
\sum_{k=0}^{m} \gamma_{k}(f_{k}(\bar{x})+C_{k})< 0
$$
which would contradict \eqref{separation} since $(f_{0}(\bar{x})+k_{0},...,f_{m}(\bar{x})+k_{m})\in\Omega_{0}$.
Thus, \eqref{separation2} holds which proves $(ii)$.
\end{proof}

In the following, we exploit \autoref{opSlemma} to get necessary and sufficient optimality conditions for general \eqref{QCQP1}.

\begin{theorem}
\label{globalminchar}
Let \autoref{preliminaryassumpt} hold and
$x^*$ be a global minimizer of \eqref{QCQP1}.
Define
 $$
 f_0(x):=J(x)-J(x^*)=x^TA_Jx+b_J^Tx-(x^*)^TA_Jx^*-2b_J^Tx^*
 $$ 
Let \autoref{omegaconv} hold, i.e. the set $\Omega_0$ defined in \eqref{omega} is convex.

Then the following Fritz-John conditions are necessary for optimality, i.e.
there exists a vector $(\gamma_0,...,\gamma_m)\in\mathbb{R}^{m+1}_+\backslash\textbf{0}_{m+1}$ such that
\begin{equation}
    \label{fritzjohn}
    \begin{split}
        &(i)\ \ \ \nabla(\gamma_0J+\sum\limits_{k=1}^m\gamma_k f_k)(x^*)=0\\
        &(ii)\ \ \gamma_kf_k(x^*)=0\ \ k\in\{1,...,m\}\\
        &(iii)\ \gamma_{0}A_J+\sum\limits_{k=1}^m\gamma_kA_k\succeq0
    \end{split}
\end{equation}
Moreover, if there exists a point $x_0\in\mathbb{R}^{n}$ such that \begin{equation} \label{slater}
f_k(x_0)<0\ \ \ \forall\,k\in\{1,...,m\}
\end{equation}
then 
there exists a vector $(\gamma_1,...,\gamma_m)\in\mathbb{R}^m_+\backslash\textbf{0}_{m}$ such that
\begin{equation}
    \label{kkt}
    \tag{KKT}
    \begin{split}
        &(i)\ \ \ \nabla(J+\sum\limits_{k=1}^m\gamma_k f_k)(x^*)=0\\
        &(ii)\ \ \gamma_kf_k(x^*)=0\ \ k\in\{1,...,m\}\\
        &(iii)\ A_J+\sum\limits_{k=1}^m\gamma_kA_k\succeq0
    \end{split}
\end{equation}
are necessary for optimality. Moreover, given a feasible $x^{*}$ for problem \eqref{QCQP1} satisfying \eqref{slater}, the conditions \eqref{kkt} are also sufficient for global optimality of $x^*$.
\end{theorem}
\begin{proof}
Let $f_0(x):=J(x)-J(x^*)$. Since $x^*$ is a global minimizer of \eqref{QCQP1}, $f_0(x)\ge0$ $\forall\,x$ feasible for \eqref{QCQP1}.
Hence, the system $f_k(x)<0$ $k=0,...,m$ has no solution.
By \autoref{opSlemma}, there exists  $(\gamma_0,...,\gamma_m)\in\mathbb{R}^{m+1}_+\backslash\textbf{0}_{m+1}$ such that 
\begin{equation} \label{alternative}
\Tilde{L}(x):=\gamma_0f_0(x)+\sum\limits_{k=1}^m\gamma_kf_k(x)\ge0\ \ \text{for all }x\in\mathbb{R}^n.
\end{equation}
In particular, for $x=x^*$, we have $\sum\limits_{k=1}^m\gamma_kf_k(x^*)\ge0$. 
Since $\gamma_kf_k(x^*)\le0$ $\forall\,k\in \{1,...,m\}$, it must be $\gamma_kf_k(x^*)=0$ $\forall\,k\in \{1,...,m\}$ which proves (ii) of \eqref{fritzjohn}.
By \eqref{alternative}, for all $x\in\mathbb{R}^n$
\begin{equation}
    \label{attmin}
        \gamma_0J(x)+\sum\limits_{k=1}^m\gamma_kf_k(x)\ge \gamma_0J(x^*)
\end{equation}
Hence $L(x):=\gamma_0J(x)+\sum\limits_{k=1}^m\gamma_kf_k(x)$ attains its minimum over $\mathbb{R}^n$ at $x^*$. 
We can apply \autoref{fermat} for a twice continuously differentiable function $L(x)$. The necessary optimality conditions $\nabla_xL(x^*)=0$ and $\nabla^2_xL(x^*)\succeq0$ are respectively equivalent to 
the conditions (i) and (iii) of \eqref{fritzjohn}, which finishes the proof of the first part.

Suppose now that \eqref{slater} holds, i.e., there exists a point $x_0$ such that $f_k(x_0)<0$ $\forall\,k=1,...,m$. 
If it were $\gamma_0=0$, then by \eqref{alternative}, it would be $\sum\limits_{k=1}^m\gamma_kf_k(x)\ge0$ for all $x\in\mathbb{R}^n$, which would contradict \eqref{slater}. Hence $\gamma_0>0$ and the Fritz-John conditions becomes the KKT condition, i.e. \eqref{kkt} holds.

To complete the proof, we show that conditions \eqref{kkt} are also sufficient for optimality. 
Assume that there exists  $x^*\in\mathbb{R}^{n}$  which is feasible to \eqref{QCQP1} and $(\gamma_1,...,\gamma_m)\in\mathbb{R}^m_+\backslash\textbf{0}_{m}$ such that \eqref{kkt} holds.
The Lagrangian for \eqref{QCQP1} is:
\begin{equation}
\label{lagqcqp}
L(x,\gamma):=J(x)+\sum\limits_{k=1}^m\gamma_kf_k(x)=x^TA(\gamma)x+2b(\gamma)^Tx+c(\gamma)
\end{equation}
with $A(\gamma)=A_J+\sum\limits_{k=1}^m\gamma_kA_k$,  $b(\gamma)=b_J+\sum\limits_{k=1}^m\gamma_kb_k$ and  $c(\gamma)=c_J+\sum\limits_{k=1}^m\gamma_kc_k$.

Notice that the Lagrangian $L(x,\gamma)$ is convex with respect to $x$, since  $A(\gamma)=A_J+\sum\limits_{k=1}^m\gamma_kA_k\succeq0$ by \eqref{kkt}. 

Hence $x^*$ such that $\nabla_x L(x^*)=0$ is the minimum of $L(x,\gamma)$ for $\gamma$ fixed, by \autoref{fermatconv}.

By \eqref{kkt}, $\gamma_kf_k(x^*)=0$ for $k=1,...,m$. We have
\begin{equation}
\label{lagmin}
J(x)+\sum\limits_{k=1}^m\gamma_kf_k(x)\ge J(x^*)+\sum\limits_{k=1}^m\gamma_kf_k(x^*)=J(x^*)\ \ \forall\,x\in\mathbb{R}^n.
\end{equation}
For any $x$ feasible for \eqref{QCQP1}, $f_k(x)\le0$ and hence 
\begin{equation}
    \label{sowehave}
    J(x)\ge J(x)+\sum\limits_{k=1}^m\gamma_kf_k(x)
\end{equation}
Combining \eqref{lagmin} and \eqref{sowehave}, for any $x$ feasible for \eqref{QCQP1}, we have
\begin{equation}
    \label{xglobmin}
    J(x)\ge J(x^*)
\end{equation}
which proves that $x^*$ is a global minimum for \eqref{QCQP1}.
\end{proof} 

\section{Global minima characterization for (S-QCQP)}
\label{global:characterisation1}

In the present section we use the results of Section \ref{global:characterisation} to provide \eqref{kkt} characterization of global minima for \eqref{QCQP2}. The main result of this section is Theorem \ref{omega0isconvex}.

Let $x^*$ be the global minimum of \eqref{QCQP2}.
As in \eqref{globalminchar}, we use the notation
 $$
 f_0(x):=J(x)-J(x^*)=a_J\|x\|^2+b_J^Tx-a_J\|x^*\|^2-2b_J^Tx^*
 $$ 

In the case of \eqref{QCQP2}, the set $\Omega_0$, defined in \eqref{omega}, take the form
\begin{equation*}
    \Omega_0:=\{(f_0(x),f_1(x),..,f_m(x))|x\in\mathbb{R}^n\}+int\mathbb{R}^{m+1}_+
\end{equation*}
where 
$$
f_k(x):=a_k\|x\|^2+2b_k^Tx+c_k\ \ k=0,...,m
$$
with $a_k\in\mathbb{R}$.

\begin{theorem}
\label{omega0isconvex}
Consider problem \eqref{QCQP2} with $m+1< n$. The set $\Omega_0$ defined in \eqref{omega} is convex.
\end{theorem}
\begin{proof}
In order to show that $\Omega_0\subset\mathbb{R}^{m+1}$ is convex, take any $v:=(v_0,...,v_m),w:=(w_0,...,w_m)\in \Omega_0$ and $\lambda\in(0,1)$. 
There exist $x_v,x_w\in\mathbb{R}^n$ such that
\begin{equation}
\label{eeeh}
    f_k(x_v)<v_k\ \ \text{and }
   f_k(x_w)<w_k\ \ \forall\,k\in\{0,...,m\}.
\end{equation}

Consider the convex combination $\lambda v + (1-\lambda)w$. Let $(\lambda v + (1-\lambda)w)_k$ be the $k$-th component of $\lambda v + (1-\lambda)w$.
By \eqref{eeeh}, we have 
\begin{equation}
    \label{fromlesstole}
    \lambda f_k(x_v)+(1-\lambda)f_k(x_w)<(\lambda v + (1-\lambda)w)_k\ \ \forall\,k\in\{0,...,m\}.
\end{equation}

In order to prove that $\Omega_0$ is convex, we show that the convex combination
$\lambda v + (1-\lambda)w$  belongs to $\Omega_0$, i.e. there exists $\Tilde{x}$ such that $\forall\,k\in\{0,...,m\}$
\begin{equation}
    \label{needtoprove}
    f_k(\Tilde{x})<(\lambda v + (1-\lambda)w)_k.
\end{equation}

Formulas \eqref{fromlesstole} and \eqref{needtoprove} together imply that $\Omega_0$ is convex if there exists $\Tilde{x}\in\mathbb{R}^n$ such that $\forall\,k\in\{0,...,m\}$
\begin{equation}
    \label{needtoprove2}
    f_k(\Tilde{x})\le\lambda f_k(x_v)+(1-\lambda)f_k(x_w).
\end{equation}

Note that if $x_v=x_w$, Then \eqref{needtoprove2} trivially holds for $\Tilde{x}=x_v=x_w$. From now on, we assume that $x_{v}$ and $x_{w}$ are distinct vectors, $x_v\neq x_w$. .
Let us take $\Tilde{x}$ such that
\begin{equation}
    \label{equality}
    \Tilde{x}\in\mathcal{S}^n:=\{x\in\mathbb{R}^n\ |\ \|x\|^2=\lambda \|x_v\|^2 + (1-\lambda)\|x_w\|^2 \}.
\end{equation}

In general, the set $\mathcal{S}^{n}$ is a sphere centered at zero with  radius $\lambda \|x_v\|^2 + (1-\lambda)\|x_w\|^2$. In particular case, when both $x_{v}=x_{w}=0$ the set $\mathcal{S}^{n}$ reduces to $\{0\}$, but this is impossible, since we assumed that $x_{v}$ and $x_{w}$ are distinct vectors.

Clearly, $\Tilde{x}$ satisfies \eqref{needtoprove} if $\forall\,k\in\{0,...,m\}$:
\begin{equation}
\label{eq23}
   \begin{split}
   &f_k(\Tilde{x})\le\lambda f_k(x_v)+(1-\lambda)f_k(x_w)<\lambda v_k+(1-\lambda)w_k.\\
    &a_k \|\Tilde{x}\|^2+2b_k^T\Tilde{x}+c_k\le a_k(\lambda \|x_v\|^2 + (1-\lambda)\|x_w\|^2)+2b_k^T(\lambda x_v + (1-\lambda)x_w)+c_k.
   \end{split}
\end{equation}

To ensure \eqref{eq23}, we show that we can choose $\Tilde{x}$ satisfying \eqref{equality} such that

\begin{equation}
\label{syyy}
   \begin{split}
       & b_k^T\Tilde{x}\le  b_k^T(\lambda x_v + (1-\lambda)x_w),\\
       &b_k^T(\Tilde{x}-(\lambda x_v + (1-\lambda)x_w))\le0.
   \end{split}
\end{equation}

Define $y:=\Tilde{x}-(\lambda x_v + (1-\lambda)x_w)\in\mathbb{R}^n$. 
Observe that $y=y(\lambda,\Tilde{x})$, where $\lambda\in(0,1)$ and $\Tilde{x}\in \mathcal{S}^{n}$. 
Clearly, the solution set of the system of inequalities
\begin{equation*}
  b_k^Ty\le 0\ \ k\in\{0,...,m\}
\end{equation*}
includes all the solutions of the system of homogeneous equations
\begin{equation}
    \label{linsystem}
    \begin{split}
        &  b_k^Ty= 0\ \ k\in\{0,...,m\}\\
        &\begin{bmatrix}b^0_1&\cdots&b^0_n\\ 
        \vdots&\ddots&\vdots\\
        b^m_1&\cdots&b^m_n\end{bmatrix}
        \begin{bmatrix}y_1\\ 
        \vdots\\
        y_n\end{bmatrix}=
        \begin{bmatrix}0\\ 
        \vdots\\
        0\end{bmatrix}
    \end{split}
\end{equation}
In the sequel we will look for  $y=\Tilde{x}-(\lambda x_v + (1-\lambda)x_w)\in\mathbb{R}^n$ among the solutions of the system of equations \eqref{linsystem}. Let
$$
B=\begin{bmatrix}b^0_1&\cdots&b^0_n\\
        \vdots&\ddots&\vdots\\
        b^m_1&\cdots&b^m_n\end{bmatrix}
$$
By assumption, $rank(B)\le m+1 < n $.
The rank-nullity theorem states that the solutions of system \eqref{linsystem} form a vector space of dimension $p=n-rank(B)$. 

Let $u_1,...,u_p$ be a basis for the vector space 
of  solutions to system \eqref{linsystem}.
Given some scalars $\alpha_i\in\mathbb{R}$, $i\in\{1,...,p\}$
a solution $y$ of \eqref{linsystem} can be written as $y=\sum\limits_{i=1}^p\alpha_iu_i$.

To complete the proof, we need to show that there exist some scalars $\alpha_i$ $i\in\{1,...,p\}$ such that
\begin{equation}
    \label{choiceoftildex}
    \Tilde{x}=\sum\limits_{i=1}^p\alpha_iu_i+\lambda x_v+(1-\lambda)x_w\in\mathcal{S}^n,
\end{equation}
where $\mathcal{S}^n$ is defined as in \eqref{equality}. 
In fact, by choosing $\Tilde{x}$ as in \eqref{choiceoftildex}, $y=\sum\limits_{i=1}^p\alpha_iu_i$ is a solution of \eqref{linsystem} and $\Tilde{x}\in\mathcal{S}^n$ satisfies \eqref{syyy}.
 We have to prove that one can choose $\alpha_{i}$, $i=1,...,p$ such that
\begin{equation}
\label{provethis}
\begin{split}
&\Tilde{x}=\sum\limits_{i=1}^p\alpha_iu_i+\lambda x_v+(1-\lambda)x_w\in\mathcal{S}^n\\
    &\|\sum\limits_{i=1}^p\alpha_iu_i+\lambda x_v+(1-\lambda)x_w\|^2=\lambda \|x_v\|^2 + (1-\lambda)\|x_w\|^2
\end{split}
\end{equation}
Below we show that we can only consider $y$ of the form $y=\alpha_{*} u_{*}$, where an index ${*}$ is chosen arbitrarily from $\{1,...,p\}.$ Indeed, given an index $*\in\{1,...,p\}$, set $\alpha_i=0$ for $i\in\{1,...,p\}\setminus\{*\}$ and $\alpha_*\neq0$. With this choice of $\alpha_{i}$, $i=1,...,p$ we can rewrite \eqref{provethis} as
$$
\tilde{x}=\alpha_{*}u_{*}+\lambda x_{v}+(1-\lambda)x_{w}
$$
and consequently, we need to find $\alpha_{*}\in\mathbb{R}$ such that
\begin{equation}
\label{provethis2}
    \begin{split}
       &\|\alpha_*u_*+\lambda x_v+(1-\lambda)x_w\|^2=\\
    &=\alpha_*^2\|u_*\|^2+\|\lambda x_v+(1-\lambda)x_w\|^2+2\alpha_*\langle u_*,\lambda x_v+(1-\lambda)x_w\rangle=\\
    &=\lambda \|x_v\|^2 + (1-\lambda)\|x_w\|^2  
    \end{split}
\end{equation}

By Corollary 2.14 of \cite{CMS}, 

\begin{equation}
    \label{corollary214}
    \|(\lambda x_v + (1-\lambda)x_w)\|^2=\lambda\|x_v\|^2 +(1-\lambda)\|x_w\|^2-\lambda(1-\lambda)\|x_v-x_w\|^2.
\end{equation}

Hence, \eqref{provethis2} becomes
\begin{equation}
    \label{endplease}
    \alpha_*^2\|u_*\|^2+2\alpha_*\langle u_*,\lambda x_v+(1-\lambda)x_w\rangle-\lambda(1-\lambda)\|x_v-x_w\|^2=0
\end{equation}

Note that by \eqref{endplease} it must be $\alpha_*\neq0$ since $x_v$ and $x_w$ are distinct vectors.
There exists $\alpha_*$ such that \eqref{endplease} holds (i.e. $\Tilde{x}\in\mathcal{S}^n$), if and only if

\begin{align*}
    &\Delta:=(\langle u_*, (\lambda x_v + (1-\lambda)x_w)\rangle)^2+\lambda(1-\lambda)\|u_*\|^2\|x_v-x_w\|^2\ge0,\\
\end{align*}

which holds for every $x_v,x_w\in\mathbb{R}^n$.

By the definition of $y$ and \eqref{choiceoftildex}, $y(\Tilde{x})=a_*u_*$ belongs to the vector space of the solutions of \eqref{linsystem}. 
Hence, \eqref{syyy} holds by choosing $\Tilde{x}$ as in \eqref{choiceoftildex}, with $\alpha_i=0$ for $i\in\{1,...,p\}\setminus\{*\}$ and $\alpha_*\neq0$ satisfying \eqref{endplease}.
\end{proof}

\begin{remark}
\begin{enumerate} 
\item  Observe that the assumption $m+1<n$ is important for the validity of the presented proof of Theorem \ref{omega0isconvex}. Otherwise, if $\text{rank } B=m+1=n$,  the only solution of the system \eqref{linsystem} is $y=0$ which implies that it must be $\tilde{x}=\lambda x_{v}+(1-\lambda)x_{w}\in\mathcal{S}^{n}$, i.e.  $\lambda x_{v}\perp (1-\lambda)x_{w}$ which can hardly be satisfied.

    \item Under the assumption of \autoref{omega0isconvex}, there could exist a component $i\in\{0,...,m\}$ such that $f_i(x):=\|x\|^2$ and a component $j\in\{0,...,m\}$ such that $f_j(x):=-\|x\|^2$.

In this case, $\Tilde{x}$ must satisfies
\begin{equation*}
    \begin{cases}
     &\|\Tilde{x}\|^2\le \lambda \|x_v\|^2 + (1-\lambda)\|x_w\|^2\\
     &-\|\Tilde{x}\|^2\le -(\lambda \|x_v\|^2 + (1-\lambda)\|x_w\|^2)
    \end{cases}
\end{equation*}

The above proves that, in order to complete the proof of  Theorem \ref{omega0isconvex}, we cannot choose $\Tilde{x}$
such that 
\begin{equation*}
    \|\Tilde{x}\|^2\neq\lambda \|x_v\|^2 + (1-\lambda)\|x_w\|^2
\end{equation*}
   This motivates our approach of looking for suitable $\tilde{x}$ from among elements of $\mathcal{S}^{n}$. 
    \item  Let us note that in the special case where $x_{v},\, x_{w}\neq 0$ and $u_{*}=\pm\ (\lambda x_{v}+(1-\lambda)x_{w})$ for some $*\in\{1,...,p\}$, the formula  \eqref{provethis} reduces to 
$$
\tilde{x}=(\alpha_{*}\pm 1)((\lambda x_{v}+(1-\lambda)x_{w})
$$
and $\tilde {x}\in \mathcal{S}^{n}$ iff 
$$
|\alpha_{*}\pm 1|=\frac{\lambda \|x_v\|^2 + (1-\lambda)\|x_w\|^2}{\|\lambda x_{v}+(1-\lambda)x_{w}\|}
$$

\end{enumerate}
\end{remark}

\autoref{omega0isconvex} allows us to prove that \eqref{kkt} conditions are necessary and sufficient optimality conditions for  \eqref{QCQP2} with $m+1<n$ under some standard assumptions, as stated in the following theorem.

\begin{theorem}
\label{sqcqpcharcter}
Consider \eqref{QCQP2} such that $m+1< n$. 
Let \autoref{preliminaryassumpt} holds and
$x^*$ be a global minimizer of \eqref{QCQP2}.
Define the matrices $H_0,H_{k}$ $k=1,...,m$ and $\Omega_0$ as in \autoref{globalminchar}.
Then the Fritz-John conditions \eqref{fritzjohn} are necessary for optimality.
Moreover, if there exists a point $x_0\in\mathbb{R}^{n}$ such that 
\begin{equation*} 
f_k(x_0)<0\ \ \ \forall\,k\in\{1,...,m\}
\end{equation*}
then 
\eqref{kkt} are necessary and sufficient for global optimality of $x^{*}$, which is feasible for \eqref{QCQP2}.

\end{theorem}

\begin{proof}
By \autoref{omega0isconvex}, the set $\Omega_0$ defined as in \eqref{omega} is convex, i.e.  \autoref{omegaconv} holds for the quadratic functionals which appear in \eqref{QCQP2}. 
We can apply \autoref{globalminchar} to complete the proof.
\end{proof}

\begin{remark}

    The condition (iii) of  \eqref{kkt}, i.e. $A_J+\sum\limits_{k=1}^m\gamma_kA_k\succeq0$ can be rewritten as $(a_J+\sum\limits_{k=1}^m\gamma_ka_k)I\succeq0$ when \autoref{assumaI} holds, i.e. for \eqref{QCQP2}.
Since all the eigenvalues of $(a_J+\sum\limits_{k=1}^m\gamma_ka_k)I$ are equal to $a_J+\sum\limits_{k=1}^m\gamma_ka_k$, (iii) takes the form  $a_J+\sum\limits_{k=1}^m\gamma_ka_k\ge0$. 
\qed
\end{remark}

\section{SDP and SOCP relaxations exactness for (S-QCQP)}
Problem \eqref{QCQP1} can be rewritten as
\begin{equation}
    \label{rewqcqp1}
    \begin{split}
        & Minimize_{ x\in \mathbb{R}^n}\, Tr(A_J,X)+2b_J^Tx\\
        &s.t.\ \ Tr(A_k,X)+2b_k^Tx+c_k\le0\,k=1,...,m\\
        &X=xx^T
    \end{split}
\end{equation}
It is possible to relax the constraint $X=xx^T$ with $X-xx^T\succeq0$, which is equivalent to the semidefinite constraint
$\begin{pmatrix}
1&x^T\\x&X
\end{pmatrix}\succeq0$.
We obtain the semidefinite or Shor relaxation for \eqref{QCQP1}, \cite{Vanden1}:
\begin{equation}
    \label{SDP}
    \tag{SDP}
    \begin{split}
        & Minimize_{ x\in \mathbb{R}^n}\, Tr(A_J,X)+2b_J^Tx\\
        &s.t.\ \ Tr(A_k,X)+2b_k^Tx+c_k\le0\,k=1,...,m\\
        &\begin{pmatrix}
1&x^T\\x&X
\end{pmatrix}\succeq0
    \end{split}
\end{equation}
Consider a diagonal \eqref{QCQP1}, which means that the matrices $A_k$ $k=J,1,...,m$ are diagonal.
Let $\alpha_k\in\mathbb{R}^n$ be the vector of all the diagonal entries of $A_k,\,\,\forall\,k$.  
Then it is possible to rewrite a diagonal \eqref{QCQP1} as follow:
\begin{equation}
    \label{benP1}
    \begin{split}
        &Minimize_{x,y\in\mathbb{R}^{n}}\ \ \alpha_J^Ty+2b_J^Tx\\
        &s.t.\ \ \alpha_k^Ty+2b_k^Tx+c_j \le 0\ \ (\forall\,k\in\{1,...,m\})\\
        &s.t.\ \ x_i^2-y_i = 0,\ \ \forall\,i
    \end{split}
\end{equation}

If we relax the constraints with $x_i^2-y_i\le0$ ($\forall\,i$),
we obtain the following convex relaxation as in \cite{ben1}:
\begin{equation}
    \label{benSDP}
    \tag{SDP2}
    \begin{split}
        &Minimize_{x,y\in\mathbb{R}^{n}}\ \ \alpha_J^Ty+2b_J^Tx\\
        &s.t.\ \ \alpha_k^Ty+2b_k^Tx+c_k \le 0\ \ (\forall\,k\in\{1,...,m\})\\
        &s.t.\ \ x_i^2-y_i \le 0,\ \ \forall\,i
    \end{split}
\end{equation}

For diagonal QCQP, \eqref{benSDP} and \eqref{SDP} are equivalent, as stated in the next theorem: 
\begin{theorem}
\label{sdpsocp}
Consider a diagonal \eqref{QCQP1}. The SDP relaxation, \eqref{SDP}, can be rewritten as the relaxation \eqref{benSDP}.
\end{theorem}
\begin{proof}
Let  $(X,x)$ be feasible for \eqref{SDP}, then $X-xx^T\succeq 0$. The principal minors of order 1 of $X-xx^T$ are $X_{ii}-x_ix_i$ $i=1,...,n$, i.e. the elements on the main diagonal of the symmetric matrix $X-xx^T$. If $X-xx^T\succeq0$ then a necessary condition for $X-xx^T$ to be positive semidefinite is $X_{ii}\ge x_ix_i=x_i^2$ $i=1,...,n$ (Sylvester criterion). 
Since $A_k\in S^n$, $k=J,1,...,m$, are diagonal then $diag(A_k)^Tdiag(X)=Tr(A_kX)\ \forall\,k$. 
So, only the elements in the main diagonal of $X$ are relevant for the objective function and the constraints of \eqref{SDP}.
We can substitute the constraint  $\begin{pmatrix}X&x\\x^T&1\end{pmatrix}\succeq0$ with $X_{ii}\ge x_i^2\ \ \forall\,i\in\{1,...,n\}$.

Taking $y\in\mathbb{R}^n$ such that $y_j= X_{jj}\ \ \forall\,j\in\{1,...,n\}$, the proof is complete. 
\end{proof}

Since \eqref{benSDP} is a convex problem, under the Slater condition, the KKT conditions are necessary and sufficient for global optimality.
Consider the \eqref{benSDP} relaxation for \eqref{QCQP2}. We have that every component of $\alpha_k$, $k=J,1,...,m$ , is equal to a constant $a_k$. For $\nu_i\ge0\,\,i=1,...,n$ and $\gamma_k\ge0\,\,k=1,...,m$, The KKT conditions for the \eqref{benSDP} relaxation for \eqref{QCQP2} are:

\begin{align}
\label{kktoap21}
        &a_J+\sum\limits_{k=1}^m\gamma_ka_k-\nu_i=0\ \ \ \forall i\in\{1,...,n\}\\
        \label{kktoap22}
        &2b_{Ji}+2\sum\limits_{k=1}^m\gamma_kb_{ki}+2\nu_ix_i=0\ \ \ \forall i\in\{1,...,n\}\\
        \label{kktoap23}
        &\gamma_k(a_k\sum\limits_{i=1}^ny_i+2\sum\limits_{i=1}^nb_{ki}x_i+c_k)=0\ \ \forall\,k\in\{1,...,m\}\\
        \label{kktoap24}
        &\nu_i(x_i^2-y_i)=0\ \ \ \forall i\\
        \label{kktoap25}
        &f_k(y,x):=a_k\sum\limits_{i=1}^ny_i+2\sum\limits_{i=1}^nb_{ki}x_i+c_k\le0\ \ \forall\,k\in\{1,...,m\}\\
        \label{kktoap26}
        &x_i^2-y_i\le 0\ \ \ \forall i\in\{1,...,n\}
\end{align}

The necessary and sufficient optimality conditions for \eqref{QCQP2} are \eqref{kkt}. They can be rewritten, for $\gamma_k\ge0\,\,k=1,...,m$,  as follow:
\begin{align}
\label{kkt21}
        &a_J+\sum\limits_{k=1}^m\gamma_ka_k\ge0\\
        \label{kkt22}
        &2b_J+2\sum\limits_{k=1}^m\gamma_kb_{k}+2[a_J+\sum\limits_{k=1}^m\gamma_ka_k]x=0\\
        \label{kkt23}
        &\gamma_k(a_k\sum\limits_{i=1}^nx_i^2+2\sum\limits_{i=1}^nb_{ki}x_i+c_k)=0\ \ \forall\,k\in\{1,...,m\}\\
        \label{kkt25}
        &f_k(x):=a_k\sum\limits_{i=1}^nx_i^2+2\sum\limits_{i=1}^nb_{ki}x_i+c_k\le0\ \ \forall\,k\in\{1,...,m\}\
\end{align}

\begin{theorem}
\label{exactness}
Consider \eqref{QCQP2} such that $m+1< n$. 
Let \autoref{preliminaryassumpt} holds and
$x^*$ be a global minimizer of \eqref{QCQP2}.
Assume that there exists $x_0$ such that $f_k(x_0)<0$ $\forall\,k\{1,...,m\}$. Then \eqref{benSDP} is exact and $(x^*,y^*)$, with $y^*_i=(x^*)^2\,\forall\,i\in\{1,...,n\}$ is a global optimum for \eqref{benSDP}.
\end{theorem}
\begin{proof}
By \autoref{globalminchar}, the conditions \eqref{kkt21}-\eqref{kkt25} are necessary and sufficient for the optimality of \eqref{QCQP2}.
Hence $x^*$ fulfills all the conditions \eqref{kkt21}-\eqref{kkt25}. 
Now consider the conditions \eqref{kktoap21}-\eqref{kktoap26}. 
By setting $x_i^2=y_i\,\forall\,i$, the system composed by \eqref{kktoap21}-\eqref{kktoap26} can be rewritten as  \eqref{kkt21}-\eqref{kkt25}. Hence, $(x^*,y^*)$, with $y^*_i=(x^*)^2\,\forall\,i$ fulfill also the KKT conditions  \eqref{kktoap21}-\eqref{kktoap26} for \eqref{benSDP}. 
Since the KKT conditions are necessary and sufficient for the optimality of \eqref{benSDP}, $(x^*,y^*)$ is an optimal solution of  \eqref{benSDP}. 
Notice that in $x^*$ the objective function of \eqref{QCQP2} yields the same value of the objective function of \eqref{benSDP} in $(x^*,y^*)$. So \eqref{benSDP} is an exact convex relaxation. 
\end{proof}

\begin{remark}
By \autoref{sdpsocp}, for a diagonal \eqref{QCQP1}, if \eqref{benSDP} is exact, then \eqref{SDP} is also exact. An analogous result was obtained in \cite{wang1}. \qed
\end{remark}

The next result can be seen as an adaptation of Theorem 3.5 from \cite{Kim1} for \eqref{QCQP2}.
\begin{theorem}
\label{socp}
Problem \eqref{benSDP} can be rewritten as the following Second Order Cone Program (SOCP)
\begin{equation}
    \label{SOCP}
    \tag{SOCP}
    \begin{split}
        &    Minimize_{\mathbf{x}\in\mathbb{R}^{3n}}\ \ \sum\limits_{i=1}^n(a_J,a_J,b_{Ji})^T\mathbf{x}_i\\
        &\sum\limits_{i=1}^n(a_k,a_k,b_{ki})^T\mathbf{x}_i+c_k\le0\ \ \forall\,k\in{1,...,m}\\
        & \mathbf{x}_i\in\mathcal{L}_i
    \end{split}
\end{equation}
where $\mathbf{x}_i=[\frac{y_i+1}{2}, \frac{y_i-1}{2}, x_i]$ and the second order cones $\mathcal{L}_i$ are defined as 
$$
\mathcal{L}_i:=\{(\frac{y_i+1}{2}, \frac{y_i-1}{2}, x_i)|\ \ x_i,y_i\in\mathbb{R}\,\text{and } \frac{y_i+1}{2}\ge\|[\frac{y_i-1}{2},x_i]\|  \}
$$
\end{theorem}

\begin{proof}
The constraints $y_i\ge x_i^2$ $i\in\{1,...,n\}$ can be rewritten as 
$$
\frac{y_i+1}{2}\ge\|[x_i, \frac{y_i-1}{2}]\|
$$
In fact, $y_i\ge x_i^2$ implies $y_i\ge0$. 
Moreover, $\frac{y_i^2}{4}-\frac{y_i}{2}+\frac{1}{4}\ge0$ $\forall\,y_i\in\mathbb{R}$ and $x_i^2\ge0$. 
Hence,
\begin{equation*}
     \frac{y_i+1}{2}=\sqrt{y_i+\frac{y_i^2}{4}-\frac{y_i}{2}+\frac{1}{4}}\ge\sqrt{x_i^2+\frac{y_i^2}{4}-\frac{y_i}{2}+\frac{1}{4}}=\|[x_i, \frac{y_i-1}{2}]\|\Leftrightarrow y_i\ge x_i^2
\end{equation*}
   
which proves that the constraints $y_i\ge x_i^2$ are equivalent to $(\frac{y_i+1}{2}, \frac{y_i-1}{2}, x_i)\in\mathcal{L}_i$ $\forall\,i\in\{1,...,n\}$.

Define the vector $\mathbf{x}\in\mathbb{R}^{3n}$ such that we can group every three components into the vector 
$$\mathbf{x}_i:=[\frac{y_i+1}{2}, \frac{y_i-1}{2}, x_i]\ \ i\in\{1,...,n\}$$

Notice that, for $k\in\{J,1,...,m\}$, we have:
$$
a_k\sum\limits_{i=1}^ny_i+2\sum\limits_{i=1}^nb_{ki}x_i=a_k\sum\limits_{i=1}^n\frac{y_i+1}{2}+a_k\sum\limits_{i=1}^n\frac{y_i-1}{2}+2\sum\limits_{i=1}^nb_{ki}x_i=\sum\limits_{i=1}^n(a_k,a_k,b_{ki})^T\mathbf{x}_i
$$
\end{proof}

\begin{remark}
\label{feasibsol}
Consider problem \eqref{QCQP2} and its convex relaxation \eqref{benSDP}.
Assume that $(\bar{x},\bar{y})$ is an optimal solution of \eqref{benSDP} and there exist $\bar{\nu}_i,\bar{\gamma}_k$ $k=1,...,m;\, i=1,...,n$ not all null, such that the KKT conditions \eqref{kktoap21}-\eqref{kktoap26} are satisfied in $(\bar{x},\bar{y},\bar{\nu},\bar{\gamma})$. 
Notice that $\bar{\nu}_i$ are all equal to $a_J+\sum\limits_{k=1}^m\gamma_ka_k$, and so we will consider just a single $\bar{\nu}\in\mathbb{R}_+$.
If $\bar{\nu}>0$, then $\bar{y}_i=\bar{x}_i^2$ $\forall\,i\in\{1,...,n\}$ and hence $\bar{x}$ is also feasible and optimal for \eqref{QCQP2} and \eqref{benSDP} is exact.\qed
\end{remark}

\section{Solving the KKT system for nonconvex S-QCQP }

In this section, we consider a S-QCQP problem of the form
\begin{equation} 
\label{problem1a}
\tag{P1}
    Minimize_{x\in \mathbb{R}^n} \ \ J(x):=1/2\|x-z\|^2\ \ s.t.\ x\in A
\end{equation}
where
\begin{equation*}
    A:=\{x\in\mathbb{R}^n\ |\ f_{k}(x):=a_k\|x\|^2+b_k^Tx+c_k\le 0,\,k=1,...,m\}\text{ and }m+1<n
\end{equation*}

Note that in order to provide a characterization of the global minima of \eqref{QCQP2}, in \autoref{globalminchar} we assumed that  a global solution $x^*\in A$ of \eqref{QCQP2} exists.
This assumption is not restrictive for \eqref{problem1a} due to the strong convexity of the objective $J$.
Hence, the problem \eqref{problem1a} is solvable.
 
 Assume that there exists $x_{0}\in A$ such that $f_{k}(x_{0})<0$ for $k=1,...,m$. By Theorem 8, the set $\Omega_0$ is convex.
 
 By Theorem \ref{globalminchar},  an element $x^{*}$ feasible for \eqref{problem1a} is a global minimizer of problem \eqref{problem1a} if and only if
conditions \eqref{kkt} hold.
Taking into account that
$$
\nabla f_{k}(x)=2a_{k}x+b_{k},\ \ k=1,...,m
$$
 \eqref{kkt} conditions for \eqref{problem1a} take the form: there exists a vector $(\gamma_1,...,\gamma_m)\in\mathbb{R}^m_+\backslash\textbf{0}_{m}$ such that
\begin{equation}
    \label{kktp1}
    \tag{KKT-P1}
    \begin{split}
        &x^{*}+\sum\limits_{k=1}^m\gamma_k(2a_{k}x^{*}+b_{k})=z\ \ \text{stationarity}\\
        &\gamma_kf_k(x^*)=0\ \ k=1,...,m \ \ \text{complementarity}\\
        &1+2\sum\limits_{k=1}^m\gamma_k a_k\ge0,\ x^{*}\in A \ \ \text{ dual and primal feasibility}
    \end{split}
\end{equation}

KKT conditions  \eqref{kktp1} can be equivalently rewritten as
\begin{equation}
    \label{kktp2}
    \begin{split}
        &(1.)\ x^{*}(1+\sum\limits_{k=1}^m\gamma_k2a_{k})=z-\sum\limits_{k=1}^m\gamma_kb_{k}\ \ \text{stationarity}\\
        &(2.)\ \gamma_k[a_{k}\|x^{*}\|^{2}+b_{k}^{T}x^{*}+c_{k}]=0\ \ k=1,...,m\ \ \text{complementarity}\\
        &(3.)\ 1+2\sum\limits_{k=1}^m\gamma_k a_k\ge0,\ x^{*}\in A \ \text{feasibility}.
    \end{split}
\end{equation}

  In the case $z=\mathbf{0}_n$, \eqref{kktp1} reduces to  finding
  $(\gamma_1,...,\gamma_m)\in\mathbb{R}^m_+\backslash\textbf{0}_{m}$
  \begin{equation}
    \label{kktp9}
    \tag{KKT-P10}
    \begin{split}
        &(1.)\ 
        \underbrace{1+2\sum\limits_{k=1}^m\gamma_k a_k}_{=w}\ge0,\ \ \text{dual feasibility}\\
        &(2.)\ f_{k}(x^{*})=a_{k}\|x^{*}\|^{2}+b_{k}^{T} x^{*}+c_{k}\le 0\ \ k=1,...,m \ \text{primal feasibility}\\
        &(3.)\ \gamma_k f_{k}(x^{*})=0 \ k=1,...,m\ \text{complementarity}.\\
          &(4.)\ \sum\limits_{k=1}^m\gamma_kb_{k}=-x^{*}w\ \text{stationarity}
    \end{split}
    \end{equation}
 
Notice that if $z\neq\mathbf{0}_n$, we can always apply conditions \eqref{kktp9} after  changing variable, i.e. taking $\bar{x}:=x-z$.
Since conditions \eqref{kktp9} are easier to handle, in the sequel we will consider always $z=\mathbf{0}_n$.
    
Consider a nonconvex and unbounded feasible set $A$ for problem \eqref{problem1a}.
We can describe the feasible set $A$ by using the terminology introduced by \cite{yang1}. Consider a concave constraints $f_i(x)$, for some $i\in\{1,...,m\}$ (we have $a_i<0$). We say that $-f_i(x)$ induces a \textit{hollow} in $\mathbb{R}^n$. 
Notice that in \cite{yang1}, it is assumed that the \textit{hollows} induced by different constraints are \textit{non intersecting}, while we do not have this assumption.

In the literature, it is not common to consider a possibly nonconvex and unbounded  feasible set as described above, if the number of constraints is greater than 2. Hence, in this section, we will compare our result with the optimality condition found in \cite{jeya1} for \eqref{QCQP1} with two constraints and in \cite{ben1} and \cite{locatelli2} for simultaneously diagonalizable (SD) \eqref{QCQP1} with $m=2$. (\eqref{QCQP1} is SD if there exists a real matrix $S$ such that the matrices $S^TA_kS$, $k=J,1...,m$ are all diagonal). Then, we solve system \eqref{kktp2} for \eqref{problem1a} with $m=2$.

Theorem 4.4. by \cite{jeya1} is the following.
\begin{theorem}(Theorem 4.4. by \cite{jeya1})
    \label{jeyam2}
    Consider a \eqref{QCQP1} such that \autoref{preliminaryassumpt} holds. 
    Assume that $m=2$, $n\ge3$ and $\exists\,\gamma_1,\gamma_2\in\mathbb{R}$ such that $\gamma_1H_1+\gamma_2H_2\succ0$, with $H$ defined as in \eqref{matrix:H}. Then the Fritz-John conditions \eqref{fritzjohn} are necessary for global optimality. Moreover, if $\exists\,x_0\in \mathbb{R}^n$ such that $f_1(x_0)<0$ and $f_2(x_0)<0$, the conditions \eqref{kkt} are necessary and sufficient for global optimality for a point $x^*$ feasible for \eqref{QCQP1}. 
\end{theorem}

\begin{remark}
For \eqref{QCQP2} with $m=2$, \autoref{jeyam2} and \autoref{sqcqpcharcter} provide the same result under different assumptions. Hence, when $m=2$, we can replace the assumption $n>m+1=3$ in \autoref{sqcqpcharcter} with $n\ge3$ and $\exists\,\gamma_1,\gamma_2\in\mathbb{R}$ such that $\gamma_1H_1+\gamma_2H_2\succ0$, i.e. 
$$
\exists\,\gamma_1,\gamma_2\in\mathbb{R} \text{ s.t. } (\gamma_1a_1+\gamma_2a_2)>0,\ \ (\gamma_1c_1+\gamma_2c_2)(\gamma_1a_1+\gamma_2a_2)-\|\gamma_1b_1+\gamma_2b_2\|^2>0.
$$
\qed
\end{remark}

Consider an SD \eqref{QCQP1} with $m=2$.
In \cite{ben1}, it is shown that if there exists a KKT point $(\bar{x},\bar{y},\bar{\gamma_1},\bar{\gamma_2},\bar{\nu}_i\ \ i=1,...,n)$ such that only one between $\bar{\gamma_1},\bar{\gamma_2}$ is strictly greater than $0$, then \eqref{benSDP} is exact. Also, it is possible to recover the global solution $x^*$ of the SD \eqref{QCQP1} with $m=2$ with simple calculations.
Paper \cite{locatelli2} takes the result of \cite{ben1} as starting point and provides the conditions such that \eqref{benSDP} is exact for SD \eqref{QCQP1} with $m=2$ and $\bar{\gamma_1}>0,\bar{\gamma_2}>0$. 
However, the following example from \cite{locatelli2} shows that the number $n$ of dimensions of the space can still be an issue.
\begin{example}(Example 3.1 from \cite{locatelli2})
Consider the following problem \eqref{QCQP2} problem.
\begin{equation}
    Minimize_{x\in\mathbb{R}}\,-x^2+x\ \ s.t.\ \ x^2\le1\ \ -x\le0
\end{equation}
Note that \eqref{benSDP} is not exact.
\end{example}

\subsection{Solving (KKT) conditions for \eqref{problem1a} with two constraints}

In the present subsection we solve  the problem of finding a global solution $x^{*}$ to \eqref{problem1a} with the help \eqref{kktp9}, by 
direct inspection of  all possible configurations $\gamma_{i}\neq 0$, $i\in I$, 
$\gamma_{i}= 0$, $i\in \{1,2\}\setminus I$, where $I\subseteq\{1,2\}$ is a subset of $\{1,2\}$. 
We are interested in the case such that $n >> 3$.

The number of all possible configuration is $4$, with 
\begin{equation*}
    I\in\{ \emptyset,\ \{1\},\ \{2\},\ \{1,2\}\}
\end{equation*}
Assume $\mathbf{0}_n$ is not feasible. Then we can exclude the case where $I=\emptyset$.

\subsubsection{$I$ is a single element set}
We start with the case where $I$ is a single element set.
Without loss of generality, set $I=\{1\}$ and $\gamma_1>0,\gamma_2=0$.  

 Conditions (1.)-(4.) of \eqref{kktp9} take the form
\begin{equation}
    \label{kktp11}
    \begin{split}
        &(1.)\  w:=1+2\gamma_1 a_1\ge0,\ \ \text{dual feasibility}\\
        &(2.)\ f_{2}(x^{*})=a_{2}\|x^{*}\|^{2}+b_{2}^{T} x^{*}+c_{2}\le 0 \text{primal feasibility}\\
        &(3.)\  f_{1}(x^{*})=a_{1}\|x^{*}\|^{2}+b_{1}^{T} x^{*}+c_{1}=0 \ \ \text{complementarity}\\
         &(4.)\ \gamma_1 b_{1}=-x^{*}w\ \text{stationarity}\\
    \end{split}
    \end{equation}
    
 
\begin{theorem}
\label{w01con}
    Let $\gamma_1>0, \gamma_2=0$.
    Let $w=0$, i.e. $\gamma_1=-\frac{1}{2a_1}>0$. If $b_{1}=0$ and there exists a solution $x^*$ of the system 
    
\begin{equation}
\label{w01consyst}
\left\{
\begin{aligned}
    &\|x^*\|^2=-\frac{c_{1}}{a_{1}}\\
& b_2^Tx^*\le a_{2}\frac{c_{1}}{a_{1}} -c_{2}
\end{aligned}
\right.
\end{equation}

Then $x^*$ is a global solution of \eqref{problem1a}.
\end{theorem}
 
\begin{proof}
Consider the case $w=0$. By (1.) of \eqref{kktp11}, if $a_{1}\ge0$, then $w=0$ is impossible. Hence, the case $w=0$ could appear only if $a_{1}<0$. Moreover,
 $$
 \|x^{*}\|^{2}=-\frac{c_{1}}{a_{1}}.
 $$
 We must have $b_{1}=0$, not to generate contradiction in $(4.)$.

The (KKT) system \eqref{kktp11} takes the form
\begin{equation}
    \label{kktp110}
    \begin{split}
        &(1.)\  w:=1+2\gamma_1 a_1=0\Rightarrow \gamma_1=-\frac{1}{2a_1}>0,\ \ \text{dual feasibility}\\
        &(2.)\ b_2^Tx^*\le a_{2}\frac{c_{1}}{a_{1}} -c_{2}\ \text{primal feasibility}\\
        &(3.)\   \|x^{*}\|^{2}=-\frac{c_{1}}{a_{1}}\ \text{complementarity}
    \end{split}
    \end{equation} 
\end{proof} 

The next remark is about the feasibility of system \eqref{w01consyst}. 
\begin{remark}
Consider the linear constraint $ b_2^Tx^*\le a_{2}\frac{c_{1}}{a_{1}} -c_{2}$. 

By assumption, $x=\mathbf{0}_n$ does not belong to the feasible set. 
Hence, if \eqref{w01consyst} is feasible, then the hyperplane 
$$ b_2^Tx= a_{2}\frac{c_{1}}{a_{1}} -c_{2},$$
must intersect or be tangent to the $n$-dimensional sphere $\|x\|^2=-\frac{c_1}{a_1}$.

We can find the least square solution $\bar{x}$ of $ b_2^Tx^*\le a_{2}\frac{c_{1}}{a_{1}} -c_{2}$.
If 
$$\|\bar{x}\|^2\le \|x^{*}\|^{2}=-\frac{c_{i}}{a_{i}}$$
then there exists a solution $x^*$ which satisfies \eqref{w01consyst}.
\qed
\end{remark}

Let us now consider the case $w>0$. 
The case $w>0$ is simpler than the case $w=0$ considered in \autoref{w01con}.
In fact, by (4.) 
$$ 
 (x^{*})^{T}x^{*}=\frac{\gamma_{i}^{2}}{w^{2}}b^{T}_{1}b_{1},
 $$
 and by the complementarity condition (3.),
\begin{equation}
\label{one-active}
 f_{1}(x^{*})=a_{1}\frac{\gamma_{1}^{2}}{w^{2}}b^{T}_{1}b_{1}-\frac{1}{w}\gamma_{1}b_{1}^{T} b_{1}+c_{1}=0
\end{equation}
$$
\Delta=\frac{(b_{1}^{T}b_{1})^{2}}{w^2}-4a_{1}c_{1}\frac{b_{1}^{T}b_{1}}{w^{2}}=\frac{b_{1}^{T}b_{1}}{w^{2}}[b_{1}^{T}b_{1}-4a_{1}c_{1}]
$$
If $[b_{1}^{T}b_{1}-4a_{1}c_{1}]\ge 0$, then \eqref{one-active}, has two (or a double) roots: $\gamma_1^1$ and $\gamma_1^2$. 
If at least one of them is positive and it is such that $w>0$, then it satisfies the dual feasibility conditions (1.) and
$$
x^{*}=\frac{-1}{w} \gamma_{1} b_{1}
$$
is a candidate for a solution to \eqref{problem1a} provided it satisfies the primal feasibility conditions  (2.) of \eqref{kktp11}.

Additionally, observe that if $a_{1}=0$, then $w=1$ and \eqref{one-active} takes the form
\begin{equation}
\label{one-active1}
c_{1}=\gamma_{1}b_{1}^{T} b_{1}, \ \ \gamma_{1}=\frac{b_{1}^{T} b_{1}}{c_{1}}.
\end{equation}
Notice that by \eqref{one-active1} $c_{1}>0$. 
Otherwise, there is no positive $\gamma_1$ which satisfies \eqref{one-active}.

\subsubsection{$I$ is not a single element set}
Consider the following theorem
\begin{theorem}
\label{gamma12pos}
Assume that $\mathbf{0}_n$ is not feasible and there does not exist a KKT point such that $I=\{1\}$ or $I=\{2\}$. Then 
there exists a (KKT) point such that $I=\{1,2\}$.
\end{theorem}
\begin{proof}
Since $\mathbf{0}_n$ is not feasible, there does not exist a KKT point such that $I=\emptyset$. By assumption there does not exist a KKT point such that $I=\{1\}$ or $I=\{2\}$. 
On the other side, an optimal point of \eqref{problem1a} must exist since the Hessian of the objective function is positive definite and the KKT conditions \eqref{kktp9} are necessary and sufficient for optimality by  \autoref{sqcqpcharcter}. 
Hence there exists a KKT point such that $I=\{1,2\}$.
\end{proof}

 Conditions (1.)-(4.) of \eqref{kktp9} take the following form. There exist $\gamma_1\ge0,\gamma_2\ge0$ not both null such that:
\begin{equation}
    \label{kktp011}
    \begin{split}
        &(1.)\  w:=1+2\gamma_1 a_1+2\gamma_2 a_2\ge0,\ \ \text{dual feasibility}\\
        &(2.)\  f_{i}(x^{*})=a_{i}\|x^{*}\|^{2}+b_{i}^{T} x^{*}+c_{i}=0 \ \ i+1,2\ \text{complementarity}\\
         &(3.)\ \gamma_1 b_{1}+\gamma_2b_2=-x^{*}w\ \text{stationarity}
    \end{split}
    \end{equation}

\begin{assumption}
\label{IND}
Assume that the vectors $b_1,b_2$ of the active constraints are linearly independent.
\end{assumption}

Observe that, whenever $n$ is very big, the \autoref{IND} is not very restrictive and it is likely to be satisfied automatically.

\begin{theorem}
    \label{wposanynumberconst1}
    Consider problem \eqref{problem1a} satisfying \autoref{IND} such that $\mathbf{0}_m$ is not feasible for \eqref{problem1a}. Assume that $a_1\neq0$.
    Let $I=\{1,2\}$ and $n>3$.
 Consider the following  system  of one quadratic polynomial and one linear equations with respect to the variables $\gamma_1,\gamma_2$.
    
    \begin{equation}
        \label{systemof1}
        \begin{aligned}
        &(1.)\ \ \begin{cases}
          &\frac{1}{a_1w}b_1^T(\gamma_1 b_1+\gamma_2 b_2)-\frac{c_1}{a_1}=\frac{1}{a_2w}b_2^T(\gamma_1 b_1+\gamma_2 b_2)-\frac{c_2}{a_2} \text{ if } a_2\neq0\\
        &\frac{1}{w}b_2^T(\gamma_1 b_1+\gamma_2 b_2)+c_2=0\text{ if } a_2=0
          \end{cases}\\
            &(2.)\ \ 
\begin{aligned}
    &\sum\limits_{i\in I}\gamma_i^2[\|b_i\|^2-2\frac{a_i}{a_1}b_1^T b_i+4\frac{c_1}{a_1}a_i^2]+\sum\limits_{i\in I} \gamma_i[4a_i\frac{c_1}{a_1}-\frac{b_1^T b_i}{a_1}]-\\
    &\hspace{4em}-2\gamma_1\gamma_2(\frac{a_2}{a_1}\|b_1\|^2+4a_1a_2)+\frac{c_1}{a_1}=0
\end{aligned}
        \end{aligned}
    \end{equation}
    
    If there exist $\gamma_1,\gamma_2> 0$, 
   not both zero, which solve \eqref{systemof1}, and 
     \begin{equation} 
     \label{wformula1}
     w=1+2\gamma_1 a_1+2\gamma_2 a_2>0
     \end{equation}
   then 
   $$x^{*}=-\frac{1}{w}(\gamma_1 b_1+\gamma_2 b_2)$$
    is a global solution of \eqref{problem1a}.
\end{theorem}

\begin{proof}
     By the fact that $0$ is not feasible for \eqref{problem1a}, \autoref{IND} implies that the condition (1.) of \eqref{kktp011} for any index set $I$ become $w>0$. Indeed, if it were $w=0$, then $0$-vector would be represented as a linear nonzero combination  of vectors $b_{1},b_2$ by (3.)of \eqref{kktp011}, which would imply that $b_{1}$ and $b_2$ are linearly dependent, contrary to \autoref{IND}.
     
          By the stationarity condition (3.) of \eqref{kktp011},
\begin{equation}
    \label{xandxsquare1}
    x^{*}=-\frac{1}{w}(\gamma_1 b_1+\gamma_2 b_2),\ \
    (x^{*})^{T}x^{*}=\frac{1}{w^{2}}(\gamma_1 b_1+\gamma_2 b_2)^{T}(\gamma_1 b_1+\gamma_2 b_2)
\end{equation}

Focus on the complementarity, (2.) of \eqref{kktp011}. Since $a_1\neq0$ by assumption, we have 
\begin{equation}
    \label{ll1}
    \begin{split}
        &f_1(x^*)=a_1\|x^*\|^2+b_1^T{x^*}+c_1=0\\
        &\|x^*\|^2=-\frac{b_1^T{x^*}}{a_1}-\frac{c_1}{a_1}\\
        &\|x^*\|^2=\frac{1}{a_1w}b_1^T(\gamma_1 b_1+\gamma_2 b_2)-\frac{c_1}{a_1}
    \end{split}
\end{equation}

Plugging \eqref{ll1} into the complementarity condition (2.), we have

\begin{equation}
    \label{ncomplementarity1}
     \begin{split}
        &(1.)\ \|x^*\|^2=\frac{1}{a_1w}b_1^T(\gamma_1 b_1+\gamma_2 b_2)-\frac{c_1}{a_1}\\
        &(2.)\ \begin{cases}
        &\frac{1}{a_1w}b_1^T(\gamma_1 b_1+\gamma_2 b_2)-\frac{c_1}{a_1}=\frac{1}{a_2w}b_2^T(\gamma_1 b_1+\gamma_2 b_2)-\frac{c_2}{a_2} \text{ if } a_2\neq0\\
        &\frac{1}{w}b_2^T(\gamma_1 b_1+\gamma_2 b_2)+c_2=0\text{ if } a_2=0
        \end{cases}
    \end{split}
\end{equation}
Equation (1.) of system \eqref{ncomplementarity1}, can be rewritten as
\begin{equation*}
    \begin{split}
    &\|x^*\|^2=\frac{1}{a_1w}b_1^T(\gamma_1 b_1+\gamma_2 b_2)-\frac{c_1}{a_1}\\
        &\frac{1}{w^{2}}(\gamma_1 b_1+\gamma_2 b_2)^{T}(\gamma_1 b_1+\gamma_2 b_2)=\frac{1}{a_1w}b_1^T(\gamma_1 b_1+\gamma_2 b_2)-\frac{c_1}{a_1}\\
        &\sum\limits_{i\in I}\gamma_i^2\|b_i\|^2+2\gamma_1\gamma_2b_1^Tb_2=\frac{w}{a_1}b_1^T(\gamma_1 b_1+\gamma_2 b_2)-w^2\frac{c_1}{a_1}\\
        &...=\frac{1}{a_1}b_1^T(\gamma_1 b_1+\gamma_2 b_2)+2\gamma_1b_1^T(\gamma_1 b_1+\gamma_2 b_2)+2\frac{a_2}{a_1}\gamma_2b_1^T(\gamma_1 b_1+\gamma_2 b_2) -w^2\frac{c_1}{a_1}\\
        &\sum\limits_{i\in I}\gamma_i^2[\|b_i\|^2-2\frac{a_i}{a_1}b_1^T b_i]-2\gamma_1\gamma_2\frac{a_2}{a_1}\|b_1\|^2 =\frac{1}{a_1}b_1^T[\sum\limits_{i\in I} \gamma_i b_i]-w^2\frac{c_1}{a_1}
    \end{split}
\end{equation*}

Moreover,

\begin{equation*}
    w^2=(1+2\gamma_1a_1+2\gamma_2a_2)^2=1+4\sum\limits_{i\in I}\gamma_i^2a_i^2+8\gamma_1\gamma_2a_1a_2+4\sum\limits_{i\in I}a_i\gamma_i
\end{equation*}

The polynomial quadratic equation becomes

\begin{equation}
\label{polyquad1}
\begin{split}
    &\sum\limits_{i\in I}\gamma_i^2[\|b_i\|^2-2\frac{a_i}{a_1}b_1^T b_i+4\frac{c_1}{a_1}a_i^2]+\sum\limits_{i\in I} \gamma_i[4a_i\frac{c_1}{a_1}-\frac{b_1^T b_i}{a_1}]-\\
    &\hspace{4em}-2\gamma_1\gamma_2(\frac{a_2}{a_1}\|b_1\|^2+4a_1a_2)+\frac{c_1}{a_1}=0
\end{split}
\end{equation}
\end{proof}

\section{Conclusions}

We provide necessary and sufficient (KKT) conditions for global optimality for a new class of possibly nonconvex quadratically constrained quadratic programming (QCQP) problems, denoted by \eqref{QCQP2}. 
Our result relies on a generalized version of the S-Lemma, \autoref{opSlemma}, stated in the context of general \eqref{QCQP1} problems. 
Hence, we believe that this paper can also be the groundwork for providing conditions for global optimality for other classes of \eqref{QCQP1}.
In section 5, we prove the exactness of the SDP and the SOCP relaxations for \eqref{QCQP2}. 
In section 6, we solve the (KKT) condition for S-QCQP when the number of constraints $m=2$ and $n >> m$.

\bibliographystyle{plain}
\bibliography{blankbib.bib}  

\end{document}